\documentclass[a4paper,11pt]{amsart}
\usepackage[english]{babel}
\usepackage[latin1]{inputenc}
\usepackage[T1]{fontenc}
\usepackage{amsmath}
\usepackage{amsthm}
\usepackage{amsfonts}
\usepackage{amssymb}
\usepackage{amsthm,bbm}
\usepackage{mathrsfs}
\usepackage{enumitem}
\usepackage{graphicx,eepic}
\usepackage{footnote}
\usepackage{hyperref}

\usepackage{amsmath}
\usepackage{amsfonts}
\usepackage{amssymb}
\usepackage{amsthm,eepic}
\usepackage{mathrsfs}
\usepackage{enumitem}
\usepackage{graphicx}

\renewcommand{\geq}{\geqslant}
\renewcommand{\leq}{\leqslant}
\newtheorem{theorem}{Theorem}

\newtheorem{rem}[theorem]{Remark}

\newtheorem{cor}[theorem]{Corollary}
\newtheorem{proposition}[theorem]{Proposition}
\newtheorem{lemma}[theorem]{Lemma}
\newtheorem{question}{Question}

\usepackage{color}
\definecolor{darkgreen}{rgb}{0,0.6,0.1}
\definecolor{MyDarkBlue}{rgb}{0,0.08,0.50}
\definecolor{BrickRed}{rgb}{0.65,0.08,0}

\hypersetup{
colorlinks=true,       % false: boxed links; true: colored links
    linkcolor=blue,          % color of internal links
    citecolor=red,      % color of links to bibliography
    filecolor=BrickRed,      % color of file links
    urlcolor=darkgreen        % color of external links
}

\title[Some aspects of fluctuations of random walks on $\mathbb R$]{Some aspects of fluctuations of random walks on $\mathbb R$ and applications to random walks on $\mathbb R^{+}$\\with non-elastic reflection at $0$}

\author[Rim Essifi]{Rim Essifi} \address{Universit\'e de Tours, F\'ed\'eration Denis Poisson, Laboratoire de Math\'ematiques et Physique Th\'eorique, Parc de Grandmont, 37200 Tours, France} \email{Rim.Essifi@lmpt.univ-tours.fr}

\author[Marc Peign\'e]{Marc Peign\'e} \address{Universit\'e de Tours, F\'ed\'eration Denis Poisson, Laboratoire de Math\'ematiques et Physique Th\'eorique, Parc de Grandmont, 37200 Tours, France} \email{Marc.Peigne@lmpt.univ-tours.fr}
  
\author[Kilian Raschel]{Kilian Raschel} \address{CNRS and F\'ed\'eration Denis Poisson, Laboratoire de Math\'ematiques et Physique Th\'eorique, Parc de Grandmont, 37200 Tours, France} \email{Kilian.Raschel@lmpt.univ-tours.fr}


\date{\today}

\begin{document}

\begin{abstract}
In this article we refine well-known results concerning the fluctuations of one-dimensional random walks. More precisely, if $(S_n)_{n \geq 0}$ is a random walk starting from $0$ and $r\geq 0$, we obtain the precise asymptotic behavior as $n\to\infty$ of $\mathbb P[\tau^{>r}=n, S_n\in K]$ and $\mathbb P[\tau^{>r}>n, S_n\in K]$, where $\tau^{>r}$ is the first time that the random walk reaches the set $]r,\infty[$, and $K$ is a compact set. Our assumptions on the jumps of the random walks are optimal. Our results give an answer to a question of Lalley stated in \cite{L}, and are applied to obtain the asymptotic behavior of the return probabilities for random walks on $\mathbb R^+$ with non-elastic reflection at $0$.
\end{abstract}

\maketitle

\section{Introduction}
\label{sec:Introduction}
\setcounter{equation}{0}
\subsection*{General context}
An essential aspect of fluctuation theory of discrete time random walks is the study of the two-dimensional renewal process formed by the successive maxima (or minima) of the random walk $(S_n)_{n \geq 0}$ and the corresponding times; this process is called the %(strictly or weakly) 
ascending (or descending) ladder process. It has been studied by many people, with major contributions by Baxter \cite{B}, Spitzer \cite{S}, and others who introduced Wiener-Hopf techniques and established several fundamental identities that relate the distributions of the ascending and descending ladder processes to the law of the random walk.

Let $(S_n)_{n \geq 0}$ be a random walk defined on a probability space  $(\Omega, \mathcal T, \mathbb P)$ and starting from $0$; in other words, $S_0=0$ and $S_n=Y_1+\cdots +Y_n$   for $n\geq 1$, where $(Y_i)_{i \geq 1}$ is a sequence of independent and identically distributed (i.i.d.)\ random variables. The strict ascending ladder process $(T_n^{*+}, H_n)_{n \geq 0}$ is defined as follows:
\begin{equation}
\label{eq:definition_ald}
     T_0^{*+}=0,\quad T_{n+1}^{*+}=\inf\{k>T_{n}^{*+}:S_k>S_{T_{n}^{*+}}\},\quad \forall n \geq 0,
\end{equation}
and
\begin{equation*}
     H_n=S_{T_n^{*+}}, \quad \forall n \geq 0.
\end{equation*}     
There exists a large literature on this process, which typically focuses on so-called {local limit theorems}, and in particular on the behavior of the probabilities $\mathbb P[T_1^{*+}>n]$ and $\mathbb P[T_1^{*+}>n,H_1\in K]$, where $K\subset \mathbb R$ is some compact set. Roughly speaking, when the variables $(Y_i)_{i \geq 1}$ admit moments of order $2$ and are centered, one  has the asymptotic behavior, as $n\to\infty$,
\begin{equation*}
     \mathbb P[T_1^{*+}>n]= \frac{a}{\sqrt{n}}(1+o(1)),\qquad 
     \mathbb P[T_1^{*+}>n, H_1\in K]= \frac{b}{n^{3/2}}(1+o(1)), 
\end{equation*}
for some constants $a, b>0$ to be specified (see for instance \cite{LP1} and references therein). 

These estimations are of great interest in several domains: one may cite for example branching processes in random environment (see for instance \cite{GK,GLL,K}) and random walks on non-unimodular groups (see \cite{LP1,LP2}); they also play a crucial role in several other less linear contexts, as in the study of return probabilities for random walks with reflecting zone on a half-line \cite{L}. 

In \cite{L}, Lalley introduced for $r>0$ the waiting time 
\begin{equation*}
     \tau^{>r}= \inf\{n >0: S_n>r\}, 
\end{equation*}
see Figure \ref{fig:tau_a}, and first looked at the behavior, as $n\to\infty$, of the probability $\mathbb P[\tau^{>r}=n, S_n \in K]$, where $K$ is a compact set. 
    \unitlength=1cm
    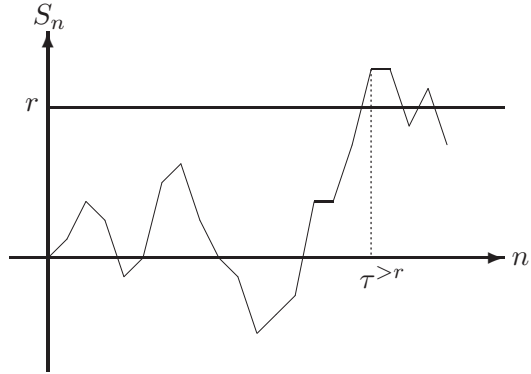
\begin{figure}[t]  
    \begin{center}
    \begin{tabular}{cccc}
    \hspace{-0.9cm}
    \begin{picture}(4.5,3)
    \thicklines
    \put(0,-1.5){{\vector(0,1){4.5}}}
    \put(-0.5,0){\vector(1,0){6.5}}
    \put(0,2){\line(1,0){6}}
    \put(-0.3,1.95){$r$}
    \put(6.1,-0.09){$n$}
    \put(-0.2,3.1){$S_n$}
    \thinlines
    \put(0,0){\line(1,1){0.25}}
    \put(0.25,0.25){\line(1,2){0.25}}
    \put(0.5,0.75){\line(1,-1){0.25}}
    \put(0.75,0.5){\line(1,-3){0.25}}
    \put(1,-0.25){\line(1,1){0.25}}
    \put(1.25,0){\line(1,4){0.25}}
    \put(1.5,1){\line(1,1){0.25}}
    \put(1.75,1.25){\line(1,-3){0.25}}
    \put(2,0.5){\line(1,-2){0.25}}
    \put(2.25,0){\line(1,-1){0.25}}
    \put(2.5,-0.25){\line(1,-3){0.25}}
    \put(2.75,-1){\line(1,1){0.25}}
    \put(3,-0.75){\line(1,1){0.25}}
    \put(3.25,-0.5){\line(1,5){0.25}}
    \put(3.5,0.75){\line(1,0){0.25}}
    \put(3.75,0.75){\line(1,3){0.25}}
    \put(4,1.5){\line(1,4){0.25}}
    \put(4.25,2.5){\line(1,0){0.25}}
    \put(4.5,2.5){\line(1,-3){0.25}}
    \put(4.75,1.75){\line(1,2){0.25}}
    \put(5,2.25){\line(1,-3){0.25}}
    %\linethickness{0.1mm}
    %\put(0,0){\dottedline{0.1}(5.25,1.5)(5.5,1)}
    \put(0,0){\dottedline{0.07}(4.25,2.5)(4.25,0)}
    \put(4.1,-0.4){$\tau^{>r}$}
    %\put(4.25,2.5){\line(0,-1){2.5}}
    \end{picture}
    \end{tabular}
    \end{center}
    \vspace{1cm}
    \caption{Definition of $\tau^{>r}$} 
    \label{fig:tau_a}
    \end{figure}
Under some strong conditions (namely, if the variables $(Y_i)_{i \geq 1}$ are lattice, bounded from above and centered), Lalley proved that 
\begin{equation}
\label{eq:asymptotic_Lalley}
     \mathbb P[\tau^{>r}=n, S_n\in K]= \frac{c}{n^{3/2}}(1+o(1)),\quad n\to\infty,
\end{equation}
for some non-explicit constant $c >0$, and wrote that ``[he] do[es] not know the minimal moment conditions necessary for [such an] estimate'' (see Equation (3.18) and below in \cite[page 590]{L}). His method is based on the Wiener-Hopf factorization and on a classical theorem of Darboux which, in this case, relates the asymptotic behavior of certain probabilities to the regularity of the underlying generating function in a neighborhood of its radius of convergence. In \cite{L}, the fact that the jumps $(Y_i)_{i \geq 1}$ are bounded from above is crucial since it allows the author to verify that the generating function of the jumps $(Y_i)_{i \geq 1}$ is meromorphic in a neighborhood of its disc of convergence, with a non-essential pole at $0$.

\subsection*{Aim and methods of this article}
In this article we obtain the asymptotic behavior of the probability in \eqref{eq:asymptotic_Lalley}, with besides an explicit formula for the constant $c$, under quite general hypotheses (Theorem \ref{MAIN=}). This in particular answers to Lalley's question. We will also obtain (Theorem \ref{MAIN>}) the asymptotic behavior of 
\begin{equation}
\label{eq:second_probability}
     \mathbb P[\tau^{>r}>n, S_n\in K],\quad n\to\infty.
\end{equation}      
To prove Theorems \ref{MAIN=} and \ref{MAIN>}, we shall adopt another strategy as that in \cite{L}, inspired by the works of Iglehart \cite{I}, Le Page and Peign\'e \cite{LP1} (Sections \ref{sec:FR} and \ref{sec:MR}). We will also propose an application of our main results to random walks on $\mathbb R^{+}$ with non-elastic reflection at $0$ (Section \ref{sec:RWabsorption}).  Finally, we shall emphasize the connections of our results with the ones of Denisov and Wachtel \cite{DW}, where quite a new approach is developed in any dimension, to find local limit theorems for random walks in cones (Section \ref{sec:DW}).

\section{First results}
\label{sec:FR}
\setcounter{equation}{0}
\subsection{Notations}
We consider here a  sequence $(Y_i)_{i \geq 1} $  of i.i.d.\ $\mathbb R$-valued random variables with law $\mu$, defined on a probability space $(\Omega, \mathcal T, \mathbb P)$.  For any $n \geq 1$, we set $\mathcal T_n= \sigma(Y_1, \ldots, Y_n)$. Let $(S_n)_{n \geq 0}$ be the corresponding random walk on $\mathbb R$  starting from $0$, i.e., $S_0=0$ and for $n\geq 1$, $S_n= Y_1+\cdots +Y_n$. In order to study  the fluctuations of $(S_n)_{n \geq0}$, we introduce for $r\in\mathbb R$ the random variables $\tau^{\geq r}$, $\tau^{>r}$, $\tau^{\leq r}$ and $\tau^{<r}$, defined by
\begin{align*}
     \tau^{\geq r}:=&\inf\{n \geq 1: S_n\geq r\},\\
     \tau^{>r}:=&\inf\{n \geq 1: S_n>r\},\\
     \tau^{\leq r}:=&\inf\{n \geq 1: S_n\leq r\},\\
     \tau^{<r}:=&\inf\{n \geq 1: S_n<r\}.
\end{align*}
Throughout we shall use the convention $\inf\{\emptyset\} = \infty$. The latter variables are stopping times with respect to the canonical filtration $(\mathcal T_n)_{n \geq 1}$. When $r=0$, in order to use standard notations, we shall rename $\tau^{\geq 0}$, $\tau^{>0}$, $\tau^{\leq 0}$ and $\tau^{<0}$ in $\tau^{+}$, $\tau^{*+}$, $\tau^{-}$ and $\tau^{*-}$, respectively. As\footnote{Here and throughout, we shall note $\mathbb R^+=[0,\infty[$, $\mathbb R^{*+}=]0,\infty[$, $\mathbb R^-=]-\infty,0]$ and $\mathbb R^{*-}=]-\infty,0[$.} $\mathbb R^{-}=\mathbb R\setminus \mathbb R^{*+}$ (resp.\ $\mathbb R^{+}=\mathbb R\setminus \mathbb R^{*-}$), there will be some duality connections between $\tau^{-}$ and $\tau^{*+}$ (resp.\ $\tau^{+}$ and $\tau^{*-}$).

We also introduce, as in \eqref{eq:definition_ald}, the sequence $(T^{*+}_n)_{n \geq 0}$ of successive ascending ladder epochs of the walk $(S_n)_{n \geq 0}$.
%they are defined by
%\begin{equation*}
%     T^{+}_0=0,\quad T^{+}_{n+1}=\inf \{k>T^{+}_{n} : S_k\geq S_{T^{+}_{n}}\},\quad \forall  n\geq 0.
%\end{equation*}      
One has $T^{*+}_1=\tau^{*+}$. Further, setting $\tau^{*+}_{n+1}:= T^{*+}_{n+1}-T^{*+}_{n}$ for any $n \geq 0$, one may write $T^{*+}_{n}= \tau^{*+}_1+\cdots+\tau^{*+}_n$, where $(\tau^{*+}_n)_{n \geq 1}$ is a sequence of i.i.d.\ random variables with the same law as $\tau^{*+}$.\footnote{Similarly, we may also consider the sequences $(T^{+}_n)_{n \geq 0}$, $(T^{-}_n)_{n \geq 0}$ and $(T^{*-}_n)_{n \geq 0}$ defined respectively by $T^{+}_0=T^{-}_0=T^{*-}_0=0$ and for $n\geq 0$, $T^{+}_{n+1}=\inf \{k>T^{+}_{n}: S_k\geq S_{T^{+}_{n}}\}$, $T^{-}_{n+1}=\inf \{k>T^{-}_{n}:S_k\leq S_{T^{-}_{n}}\}$ and $T^{*-}_{n+1}=\inf \{k>T^{*-}_{n}:S_k< S_{T^{*-}_{n}}\}$.}

\subsection{Hypotheses} Throughout this manuscript, we shall assume that the law $\mu$ satisfies one of the following moment conditions {\bf M}:
\begin{itemize}
     \item[{\bf M($k$):}]  {\it $\mathbb E[|Y_1|^k]<\infty$;}
     \item[{\bf M($\exp$):}] {\it $\mathbb E[\exp(\gamma Y_1)]<\infty$, for all $\gamma \in \mathbb R$;}
     \item[{\bf M($\exp^{-}$):}] {\it $\mathbb E[\exp(\gamma Y_1)]<\infty$, for all $\gamma \in \mathbb R^{-}$.}
\end{itemize}
We shall also often suppose 
\begin{enumerate}
     \item[{\bf C:}] {\it $\mathbb E[Y_1]=0$}.
\end{enumerate}
 
Under {\bf M($1$)} and {\bf C}, the variables $\tau^{+}$, $\tau^{*+}$, $\tau^{-}$ and $\tau^{*-}$ are $\mathbb P$-a.s.\ finite, see \cite{F2},\footnote{Notice that this property also holds for symmetric laws $\mu$ without any moment assumption.} and we denote by $\mu^{+} $ (resp.\ $\mu^{*+}, \mu^-, \mu^{*-}$) the law of the variable $S_{\tau^{+}}$ (resp.\ $S_{\tau^{*+}}, S_{\tau^{-}}$ and $S_{\tau^{*-}}$). %It is worth mentioning that this assumption is optimal (in fact, the constant $\sigma^2:= \mathbb E[Y_1^2]$ will appear explicitly in all the results hereafter).%\footnote{{\bf Ici, j'ai retire une phrase sur le fait que dans \cite{DW} il y avait une condition $2+\epsilon$. Si vous preferez on peut la remettre.}} We finally notice that the main statements in \cite{DW} are given under the stronger assumption of existence of moments of order $2+\epsilon$; however, their approach is valid in any dimension.\footnote{{\bf A la relecture, je ne suis plus sur que ca vaille la peine d'ecrire cette phrase a cet endroit; c'est un peu gratuit.}}

We will also consider the two following couples of hypotheses {\bf AA}:
\begin{enumerate}
\item[{\bf AA($\mathbb Z$):}]  {\it the measure $\mu$ is {adapted} on $\mathbb Z$ (i.e., the group generated by the support $S_\mu$ of $\mu$ is equal to $\mathbb Z$)  and {aperiodic} (i.e., the group generated by $S_\mu-S_\mu$ is equal to $\mathbb Z$);}
 \item[{\bf AA($\mathbb R$):}] {\it the measure $\mu$ is {adapted} on $\mathbb R$ (i.e., the closed group generated by the support $S_\mu$ of $\mu$ is equal to $\mathbb R$) and {aperiodic} (i.e., the closed group generated by $S_\mu-S_\mu$ is equal to $\mathbb R$).}
\end{enumerate}

\subsection{Classical results} Let us now recall the result below, which concerns the probability \eqref{eq:second_probability} for $r=0$. 

\begin{theorem}[\cite{I,LP1}]
\label{prop:classical_result} 
Assume that the hypotheses {\bf AA}, {\bf C} and {\bf M($2$)} hold. Then for any continuous function $\phi $ with compact support on  $\mathbb R$, one has\footnote{Below and throughout, for any bounded random variable $Z: \Omega \to\mathbb R$ and any event $A \in \mathcal T$, one sets $\mathbb E[A;Z]:= \mathbb E[Z\mathbbm{1}_{A}]$.}
\begin{equation*}
     \lim_{n \to \infty} n^{3/2} \mathbb E[ \tau^{*+}>n; \phi(S_n)]= a^{-}(\phi):= \int_{\mathbb R^-} \phi(t) a^{-}({\rm d}t):= \frac{1}{\sigma\sqrt{2\pi}} \int_{\mathbb R^-} \phi(t)  \lambda^{-}*U^{-}({\rm d}t),
\end{equation*}
where  
\begin{itemize}
     \item $\sigma^2:= \mathbb E[Y_1^2]$;
     \item $\lambda^-$ is the counting measure on $\mathbb Z^-$  when  {\rm {\bf AA}($\mathbb Z$)} holds  (resp.\ the Lebesgue measure on $\mathbb R^-$ when  {\rm {\bf AA}($\mathbb R$)} holds);\footnote{For an upcoming use, we also introduce 
\begin{itemize} 
\item[$\bullet$] the counting measures $\lambda^{*-}$, $\lambda^+$ and $\lambda^{*+}$ on $\mathbb Z^{*-}$, $\mathbb Z^{+}$ and $\mathbb Z^{*+}$, respectively; 
\item[$\bullet$] the Lebesgue measures $\lambda^{*-}$, $\lambda^+$ and $\lambda^{*+}$ on $\mathbb R^{*-}$, $\mathbb R^{+}$ and $\mathbb R^{*+}$, respectively. 
\end{itemize}  
Notice that $\lambda^{*-}= \lambda^{-}$ and $\lambda^{*+}=\lambda^{+}$ when {\bf AA}($\mathbb R$) holds, but we keep the two notations in order to unify the statements under the two types of hypotheses {\bf AA}.}
     \item $  U^{-}$ is  the  $\sigma$-finite potential $U^-:= \sum_{n\geq0} (\mu^-)^{*n}$.
\end{itemize}
\end{theorem}
Since some arguments will be quite useful and used in the sequel, we give below a sketch of the proof of Theorem \ref{prop:classical_result}, following \cite{I,LP1}. By a standard argument in measure theory (see Theorem 2 in Chapter XIII on Laplace transforms in the book \cite{F2}), it is sufficient to prove the above convergence for all functions $\phi$ of the form $\phi(t)=\exp(\alpha t)$, $\alpha >0$ (indeed, notice that the support of the limit measure $a^{-}({\rm d}t)$ is included in $\mathbb R^-$). We shall use the same remark when proving Theorems \ref{prop2} and \ref{MAIN=}.

\begin{proof}[Sketch of the proof of Theorem \ref{prop:classical_result} in the case {\rm{\bf AA}($\mathbb Z$)}] 
We shall use the following identity, which is a consequence of the Wiener-Hopf factorization (see \cite[P5 in page 181]{S}):
\begin{equation}
\label{identity1}
     \phi_\alpha(s):= \sum_{n\geq 0}s^n \mathbb E[ \tau^{*+}>n; e^{\alpha S_n}]= \exp B_\alpha(s),
     \quad \forall s \in [0, 1[,\quad \forall \alpha  >0,
\end{equation}
where 
\begin{equation*}
     B_\alpha(s):=\sum_{n\geq1}\frac{s^n}{n}\mathbb E[ S_n\leq 0; e^{\alpha S_n}].
\end{equation*}
Further, by the classical local limit theorem on $\mathbb Z$ (this is here that we use {\bf M($2$)}, see for instance \cite[P10 in page 79]{S}), one gets
\begin{equation*}
     \mathbb E[ S_n\leq 0; e^{\alpha S_n}]=\frac{1}{\sigma\sqrt{2\pi n}} \frac{1}{1-e^{-\alpha}}(1+o(1)),
     \quad n\to\infty.
\end{equation*}     
Accordingly, the sequence $(n^{3/2}\mathbb E[ \tau^{*+}>n; e^{\alpha S_n}])_{n \geq 1}$ is bounded, thanks to Lemma \ref{Bruijn1} below (taken from \cite[Lemma 2.1]{I}), applied with $b_n:= \mathbb E[ S_n\leq 0; e^{\alpha S_n}]/n$ and $d_n:= \mathbb E[ \tau^{*+}>n; e^{\alpha S_n}]$.

\begin{lemma}[\cite{I}]
\label{Bruijn1}
Let $\sum_{n\geq0}d_ns^n=\exp\sum_{n\geq0}b_ns^n$. If the sequence $(n^{3/2}b_n)_{n \geq 1}$ is bounded, the same holds for $(n^{3/2}d_n)_{n \geq 1}$.
\end{lemma}
Differentiating the two members of \eqref{identity1} with respect to $s$, one gets 
\begin{equation*}
\label{identity1'}
     \phi'_\alpha(s)=\sum_{n\geq 1}ns^{n-1} \mathbb E[ \tau^{*+}>n; e^{\alpha S_n}]=\phi_\alpha(s)\sum_{n\geq1}{s^{n-1} }\mathbb E[ S_n\leq 0; e^{\alpha S_n}].
\end{equation*}
We then make use of Lemma \ref{Bruijn2} (see \cite[Lemma 2.2]{I} for the original statement), applied with $c_n:=\mathbb E[ S_n\leq 0; e^{\alpha S_n}]=nb_n$, $d_n:= \mathbb E[ \tau^{*+}>n; e^{\alpha S_n}]$ and $a_n:= n \mathbb E[ \tau^{*+}>n; e^{\alpha S_n}]$.

\begin{lemma}[\cite{I}]
\label{Bruijn2}
Let $(c_n)_{n \geq 0}$ and $(d_n)_{n \geq 0}$ be sequences of non-negative real numbers such that
\begin{enumerate}
\item $\lim_{n \to \infty} \sqrt{n} c_n = c>0$;
\item $\sum_{n\geq0} d_n = D<\infty$;
\item $(nd_n)_{n \geq 0}$ is bounded.
\end{enumerate}
If $a_n= \sum_{0\leq k\leq n-1} d_kc_{n-k}$, then $\lim_{n\to \infty} \sqrt{n} a_n = cD$.
\end{lemma}
This way, one reaches the conclusion that
\begin{equation*}
     \lim_{n \to \infty}n^{3/2} \mathbb E[ \tau^{*+}>n; e^{\alpha S_n}]= \frac{1}{\sigma\sqrt{2\pi}}\frac{1}{1-e^{-\alpha}} \sum_{n\geq 0}\mathbb E[ \tau^{*+}>n; e^{\alpha S_n}].
\end{equation*}
To conclude, it remains to express differently the limit. First, the factor $1/({1-e^{-\alpha}})$ is equal to $\int_{\mathbb R } e^{\alpha t} \lambda^-({\rm d}t)$. Further, since the vectors $(Y_1, \ldots, Y_n)$ and $(Y_n, \ldots, Y_1)$ have the same law, one gets
\begin{eqnarray*}
\sum_{n\geq 0}\mathbb E[ \tau^{*+}>n; e^{\alpha S_n}]&=&
\sum_{n\geq 0}\mathbb E[ S_1\leq 0, S_2\leq 0,\ldots, S_n\leq 0; e^{\alpha S_n}] \\
&=&
\sum_{n\geq 0}\mathbb E[ S_n\leq S_{n-1}, S_n\leq S_{n-2},    \ldots, S_n\leq 0; e^{\alpha S_n}] \\
&=&
\sum_{n\geq 0}\mathbb E[ \exists \ell \geq 0: T_\ell^{-}=n; e^{\alpha S_n}] \\
&=&
\sum_{\ell\geq 0}\mathbb E[ e^{\alpha S_{T_\ell^{-}}}]=  U^-(x\mapsto e^{\alpha x}),
\end{eqnarray*}
i.e.,  $
\sum_{n\geq 0}\mathbb E[ \tau^{*+}>n;   S_n\in {\rm d}x]
= U^-({\rm d}x)$,
so that 
\begin{equation*}
     \frac{1}{\sigma\sqrt{2\pi}} \frac{1}{1-e^{-\alpha}} \sum_{n\geq 0}\mathbb E[ \tau^{*+}>n; e^{\alpha S_n}]= \frac{1}{\sigma\sqrt{2\pi}}\int_{\mathbb R^-} e^{\alpha t}\lambda^{-}*U^{-}({\rm d}t).
\end{equation*}
The proof is complete.
\end{proof}

\begin{rem}
\label{rem-1}
For similar reasons as in the proof of Theorem \ref{prop:classical_result}, one has
\begin{align*}
     &\textstyle\sum_{n\geq 0}\mathbb E[ \tau^{+}>n;   S_n\in {\rm d}x]= U^{*-}({\rm d}x):=\sum_{n \geq 0}(\mu^{*-})^{*n}({\rm d}x),\\
     &\textstyle\sum_{n\geq 0}\mathbb E[ \tau^{*-}>n;   S_n\in {\rm d}x]= U^{+}({\rm d}x):=\sum_{n \geq 0}(\mu^{+})^{*n}({\rm d}x),\\
     &\textstyle\sum_{n\geq 0}\mathbb E[ \tau^{-}>n;   S_n\in {\rm d}x]= U^{*+}({\rm d}x):=\sum_{n \geq 0}(\mu^{*+})^{*n}({\rm d}x),
\end{align*}    
as well as the weak convergences, as $n \to \infty$,
\begin{eqnarray*}
n^{3/2}\mathbb E[\tau^{*+}>n; S_n\in {\rm d}x ]\quad {\longrightarrow}& \phantom{{}^*}a^-({\rm d}x)&\hspace{-3mm}:= (1/\sigma \sqrt{2\pi})\lambda^{-}*U^-,\\
n^{3/2}\mathbb E[\tau^{+}>n; S_n\in {\rm d}x ]\quad {\longrightarrow}& a^{*-}({\rm d}x)&\hspace{-3mm}:=(1/\sigma \sqrt{2\pi})\lambda^{*-}*U^{*-},\\
n^{3/2}\mathbb E[\tau^{*-}>n; S_n\in {\rm d}x ]\quad {\longrightarrow}& \phantom{{}^*}a^+({\rm d}x)&\hspace{-3mm}:= (1/\sigma \sqrt{2\pi})\lambda^{+}*U^+,\\
n^{3/2}\mathbb E[\tau^{-}>n; S_n\in {\rm d}x ]\quad {\longrightarrow}& a^{*+}({\rm d}x)&\hspace{-3mm}:= (1/\sigma \sqrt{2\pi})\lambda^{*+}*U^{*+}.
\end{eqnarray*}
\end{rem}

We conclude this part by finding the asymptotic behavior of $\mathbb P[\tau^{*+}>n]$. Using the well-known expansion 
\begin{equation*}
     \sqrt{1-s}=\exp\left(\frac{1}{2}\ln(1-s)\right)=\exp\left(-\frac{1}{2}\sum_{n\geq 1}\frac{s^n}{n}\right)
\end{equation*}
and setting $\alpha=0$ in \eqref{identity1}, one gets that for $s$ close to $1$,
\begin{equation*}
\label{eq:}
     \sum_{n\geq 0}s^n \mathbb P[ \tau^{*+}>n]= \exp \left(\sum_{n\geq1}\frac{s^n}{n}\mathbb P[ S_n\leq 0]\right)=\frac{\exp \kappa}{\sqrt{1-s}}(1+o(1)),
\end{equation*}
where
\begin{equation}
\label{eq:def_kappa}
     \kappa = \sum_{n\geq1}\frac{1}{n}\left(\mathbb P[ S_n\leq 0]-\frac{1}{2}\right).
\end{equation}
Notice that the series in \eqref{eq:def_kappa} is absolutely convergent, see \cite[Theorem 3]{R}.\footnote{There also exists the following expression for $\kappa$: $e^\kappa = (\sqrt{2}/\sigma)\mathbb E[S_{\tau^{*+}}]$, see \cite[P5 in Section 18]{S}.\label{footnote:Rim}} By a standard Tauberian theorem, since the sequence $(\mathbb P[\tau^{*+}>n])_{n\geq 0}$ is decreasing, one obtains (see \cite{LP1})%\footnote{\bf Bonne reference ?}
\begin{equation}
\label{eq:ab_eps}
     \mathbb P[\tau^{*+}>n] = \frac{\exp \kappa}{\sqrt{\pi n}}(1+o(1)),\quad n\to\infty.
\end{equation}
Note that the monotonicity of the sequence $(\mathbb P[\tau^{*+}>n])_{n\geq 0}$ is crucial to replace the Ces\`aro means convergence by the usual convergence.

\subsection{Extensions}

Equation \eqref{eq:ab_eps} shows that the asymptotic behavior of $\mathbb P[\tau^{*+}>n]$ is in $1/\sqrt{n}$ as $n\to \infty$. As for the probability $\mathbb P[\tau^{*+}=n]$, we have the following result, which is proved in \cite{AD,E}.

\begin{proposition}
\label{SECOND=0}
Assume that the hypotheses {\bf AA}, {\bf C} and {\bf M($2$)} hold. Then the sequence $(n^{3/2}\mathbb P[\tau^{*+}=n])_{n\geq 0}$ converges to some positive constant.
\end{proposition}

We now refine Proposition \ref{SECOND=0}, by adding in the probability the information of the position of the walk at time $\tau^{*+}$. Using the same approach as for Theorem \ref{prop:classical_result}, we may obtain the following theorem, which we did not find in the literature:
\begin{theorem}
\label{prop2} 
Assume that the hypotheses {\bf AA}, {\bf C} and {\bf M($2$)} hold. Then for any continuous function $\phi $ with compact support on  $\mathbb R$, one has
\begin{equation*}
     \lim_{n \to \infty} n^{3/2} \mathbb E[ \tau^{*+}=n; \phi(S_n)]=b^{*+}(\phi):=\int_{\mathbb R^+} \phi(t) b^{*+}({\rm d}t):= \frac{1}{\sigma\sqrt{2\pi}}\int_{\mathbb R^+} \phi(t)  \lambda^{*+}*\mu^{*+}({\rm d}t ),
\end{equation*}
where $\lambda^{*+}$ is  the counting measure on $\mathbb Z^{*+}$  when  {\rm {\bf AA}($\mathbb Z$)} holds (resp.\ the Lebesgue measure on $\mathbb R^{*+}$ when  {\rm {\bf AA}($\mathbb R$)} holds).
\end{theorem}

\begin{proof}[Sketch of the proof of Theorem \ref{prop2} in the case {\rm{\bf AA}($\mathbb Z$)}] 
We shall use the following identity, which as \eqref{identity1} is a consequence of the Wiener-Hopf factorization:
\begin{equation}
\label{identity2}
     \psi_\alpha(s):=   \sum_{n\geq 0}s^n \mathbb E[ \tau^{*+}=n; e^{-\alpha S_n}]= 1-\exp - \widetilde B_\alpha(s),\quad 
     \forall s \in [0, 1[,\quad \forall \alpha >0,
\end{equation}
where
\begin{equation*}
     \widetilde B_\alpha(s):=\sum_{n\geq1}\frac{s^n}{n}\mathbb E[ S_n> 0; e^{-\alpha S_n}].
\end{equation*}
Setting $d_n:=  \mathbb E[ \tau^{*+}=n; e^{-\alpha S_n}]$, the same argument as in the proof of Theorem \ref{prop:classical_result} (via Lemma \ref{Bruijn1}) implies that the sequence $(n^{3/2} d_n)_{n \geq 1}$ is bounded (we notice that in Lemma \ref{Bruijn1},  the sequences $(b_n)_{n \geq 0}$ and $(d_n)_{n \geq 0}$  are not necessarily non-negative, so it can be applied in the present situation).
 
Differentiating the two members of \eqref{identity2} with respect to $s$ then yields
\begin{equation*}
\label{identity1''}
     \psi'_\alpha(s)=\sum_{n\geq 1}
ns^{n-1} \mathbb E[ \tau^{*+}=n; e^{-\alpha S_n}]=  (1-\psi_\alpha(s))\sum_{n\geq1}{s^{n-1} }\mathbb E[ S_n>0; e^{-\alpha S_n}],
\end{equation*}
and Theorem \ref{prop2} is thus a consequence of Lemma \ref{Bruijn2}, applied with $c_n:=\mathbb E[ S_n>0; e^{-\alpha S_n}]$, $d_n:= \mathbbm{1}_{\{n=0\}}- \mathbb E[ \tau^{*+}=n; e^{-\alpha S_n}]$ and $a_n:= n \mathbb E[ \tau^{*+}=n; e^{-\alpha S_n}]$.
\end{proof}

According to the previous proof, we also have, as $n \to \infty$, the weak convergences below:
\begin{eqnarray*}
n^{3/2}\mathbb E[\tau^{*+}=n; S_n\in {\rm d}x ]\quad{\longrightarrow}& b^{*+}({\rm d}x)&\hspace{-3mm}:= (1/\sigma \sqrt{2\pi})\lambda^{*+}*\mu^{*+},\\
n^{3/2}\mathbb E[\tau^{+}=n; S_n\in {\rm d}x ]\quad{\longrightarrow}& \phantom{{}^*}b^{+}({\rm d}x)&\hspace{-3mm}:=(1/\sigma \sqrt{2\pi})\lambda^{+}*\mu^{+},\\
n^{3/2}\mathbb E[\tau^{*-}=n; S_n\in {\rm d}x ]\quad {\longrightarrow}& b^{*-}({\rm d}x)&\hspace{-3mm}:= (1/\sigma \sqrt{2\pi})\lambda^{*-}*\mu^{*-},\\
n^{3/2}\mathbb E[\tau^{-}=n; S_n\in {\rm d}x ]\quad {\longrightarrow}& \phantom{{}^*}b^{-}({\rm d}x)&\hspace{-3mm}:=(1/\sigma \sqrt{2\pi})\lambda^{-}*\mu^{-}.
\end{eqnarray*}

\section{Main results}
\label{sec:MR}
\setcounter{equation}{0}

In this section we are first interested in the expectation $\mathbb E[ \tau^{>r}=n; \phi(S_n)]$, for any fixed value of $r>0$. In Theorem \ref{MAIN=} we find its asymptotic behavior as $n\to\infty$, for any continuous function $\phi $ with compact support on $\mathbb R$. Then in Proposition \ref{SECOND=} we take $\phi$ identically equal to $1$, and we prove that the sequence $(n\mathbb P[\tau^{>r}=n])_{n\geq 0}$ is bounded. We then consider the expectation $\mathbb E[ \tau^{>r}>n; \phi(S_n)]$. We first derive its asymptotic behavior as $n\to\infty$, in Theorem \ref{MAIN>}. Finally, in Proposition \ref{SECOND>} we obtain the asymptotics of the probability $\mathbb P[ \tau^{>r}>n]$ for large values of $n$. The theorems stated in Section \ref{sec:MR} concern the hitting time $\tau^{>r}$; similar statements (obtained exactly along the same lines) exist for the hitting times $\tau^{\geq r}$, $\tau^{< r}$ and $\tau^{\leq r}$.

\begin{theorem}
\label{MAIN=}
Assume that the hypotheses {\bf AA}, {\bf C} and {\bf M($2$)} hold. Then for any continuous function $\phi $ with compact support on $]r,\infty[$, one has
\begin{equation*}
     \lim_{n \to \infty} n^{3/2} \mathbb E[ \tau^{>r}=n; \phi(S_n)]= \iint_{\Delta_r} \phi(x+y) U^{*+}({\rm d}x)b^{*+}({\rm d}y)+ \iint_{\Delta_r} \phi(x+y)a^{*+}({\rm d}x)\mu^{*+}({\rm d}y),
\end{equation*}     
where 
$\Delta_r:= \{(x, y) \in \mathbb R^{*+}\times \mathbb R^{*+}: 0\leq x \leq r, x+y>r\}$.
\end{theorem}

\begin{proof} 
Since $\phi$ has compact support in $]r, \infty[$, one has
\begin{eqnarray*}
\mathbb E [\tau^{>r}=n; \phi(S_n)] &=& \sum_{0\leq k\leq n}
\mathbb E [\exists \ell \geq 0, T^{*+}_\ell=k, S_k\leq r, n-k=\tau^{*+}_{\ell+1}, S_n>r;   \phi(S_n)] \\
&=&\sum_{0\leq k\leq n}\iint_{\Delta_r} \phi(x+y)\mathbb P [\exists \ell \geq 0, T^{*+}_\ell=k, S_k\in {\rm d}x]\times 
\\& & \qquad\qquad\qquad\qquad\qquad\qquad\qquad\times \mathbb P [\tau^{*+}=n-k, S_{n-k}\in {\rm d}y]
\\
&=&\sum_{0\leq k\leq n} I_{n, k} (r, \phi),
\end{eqnarray*}
where we have set 
\begin{equation}
\label{eq:def_I_n_k}
     I_{n, k} (r, \phi):= \iint_{\Delta_r} \phi(x+y)\mathbb P [ \tau^{-}>k, S_k\in {\rm d}x ] \mathbb P[\tau^{*+}=n-k, S_{n-k}\in {\rm d}y ].
\end{equation}     
In Equation \eqref{eq:def_I_n_k} above, we have used the equality $\mathbb P[\exists \ell \geq 0,T_\ell^{*+}=k,S_k\in {\rm d}x ]=\mathbb P[\tau^{-}>k,S_k\in {\rm d}x ]$. It follows by the same arguments as in the proof of Theorem \ref{prop:classical_result} (below Lemma \ref{Bruijn2}). To pursue the proof, we shall use the following elementary result (see \cite[Lemma II.8]{LP1} for the original statement and its proof):

\begin{lemma}
\label{LP1}
Let $(a_n)_{n \geq 0}$ and $(b_n)_{n \geq 0}$ be two sequences of non-negative real numbers such that $\lim_{n \to \infty} n^{3/2} a_n= a \in \mathbb R^{*+}$ and $\lim_{n \to \infty} n^{3/2} b_n= b \in \mathbb R^{*+}$. Then:
\begin{itemize}
\item there exists $C>0$ such that, for any $n \geq 1$ and any $0<i<n-j<n$, 
\begin{equation*}
     n^{3/2}\sum_{i+1\leq k\leq n-j}a_k b_{n-k} \leq C \left({1\over \sqrt{i}}+{1\over \sqrt{j}}\right);
\end{equation*}
\item setting $A:=  \sum_{n\geq 0} a_n$ and $B:=  \sum_{n\geq 0} b_n$, one has 
\begin{equation*}
     \lim_{n \to \infty} n^{3/2} \sum_{k=0}^{n}a_k b_{n-k}= a B+b A.
\end{equation*}
\end{itemize}
\end{lemma}

Since $\phi$ is non-negative with compact support in $]r, \infty[$, there exists a constant $c_\phi>0$ such that $\phi(t) \leq c_\phi e^{-t}$, for all $t \geq 0$. This yields that for any $0<i<n-j<n$,
\begin{equation*}
     \sum_{i+1\leq k\leq n-j} I_{n, k} (r, \phi)  \leq  c_\phi  \sum_{i+1\leq k\leq n-j}a_k b_{n-k},
\end{equation*}
with $a_k:= \mathbb E [ \tau^{-}>k; e^{-S_k} ]$ and $b_k:= \mathbb E [ \tau^{*+}=k;e^{-S_k} ]$. With Lemma \ref{LP1} we deduce that there exists some constant $C>0$ such that
\begin{equation*}
     \sum_{i+1\leq k\leq n-j}  I_{n, k} (r, \phi)\leq    C  \left({1\over \sqrt{i}}+{1\over \sqrt{j}}\right).
\end{equation*}

On the other hand,  for any fixed $k\geq 1$ and $x \in [0,r]$, one has by Theorem \ref{prop2} 
\begin{equation*}
     \lim_{n \to \infty}n^{3/2} \int_{\{y\geq 0\}} \phi(x+y) \mathbb P [ \tau^{*+}=n-k, S_{n-k}\in {\rm d}y ]= \int_{\{y\geq 0\}} \phi(x+y)b^{*+}({\rm d}y).
\end{equation*}
Further, for any $k\geq 1$, the function
\begin{equation*}
     x\mapsto  n^{3/2} \int_{\{y\leq 0\}} \phi(x+y) \mathbb P[ \tau^{*+}=n-k, S_{n-k}\in {\rm d}y ]
\end{equation*}      
is dominated on $[0, r]$ by 
$
     x\mapsto c_\phi (\sup_{n\geq1} n^{3/2}\mathbb E [ \tau^{*+}=n-k;e^{-S_{n-k}}]) e^{-x},
$
which is bounded (by Theorem \ref{prop2}) and so integrable with respect to the measure $\mathbb P[ \tau^{-}>k, S_k\in {\rm d}x]$. The dominated convergence theorem thus yields 
\begin{equation*}
     \lim_{n\to \infty} n^{3/2}  \sum_{0\leq k\leq i}I_{n, k} (r, \phi) =   \sum_{0\leq k\leq i} \iint _{\Delta_r}\phi(x+y) \mathbb P[\tau^{-}>k, S_k\in {\rm d}x ]  b^{*+}({\rm d}y).
\end{equation*}  
The same argument leads to 
\begin{equation*}
     \lim_{n\to \infty} n^{3/2}  \sum_{n-j\leq k\leq n}I_{n, k} (r, \phi) =   \sum_{0\leq k\leq j}\iint _{\Delta_r}\phi(x+y)  a^{*+}({\rm d}x )\mathbb P [ \tau^{*+}=k, S_k\in {\rm d}y].
\end{equation*}  
Letting $i, j \to \infty$ and using the equalities 
\begin{equation*}
     \sum_{k\geq 0}\mathbb E[ \tau^{-}>k;   S_k\in {\rm d}x]= U^{*+}({\rm d}x),\qquad
\sum_{k\geq 0}\mathbb E[ \tau^{*+}=k;   S_k\in {\rm d}y]= \mu^{*+}({\rm d}y),
\end{equation*}  
one concludes.
\end{proof}

\begin{proposition}
\label{SECOND=}
Assume that the hypotheses {\bf AA}, {\bf C} and {\bf M($2$)} hold. Then for any $r\in\mathbb R^{+}$, the sequence $(n\mathbb P[\tau^{>r}=n])_{n\geq 0}$ is bounded.
\end{proposition}

\begin{proof}
By the proof of Theorem \ref{MAIN=}, one may decompose $\mathbb P[\tau^{>r}=n]$ as $\sum_{0\leq k\leq n} I_{n,k}( r,1 )$, with $I_{n,k}$ defined in \eqref{eq:def_I_n_k}. One easily obtains that
\begin{equation*}
     I_{n,k}( r,1 )\leq \mathbb P [ \tau^{-}>k, S_k\in [0,r] ] \mathbb P[\tau^{*+}=n-k],
\end{equation*}
$\Delta_r$ being defined as in Theorem \ref{MAIN=}. One concludes by applying Remark \ref{rem-1} (we obtain the estimate $1/k^{3/2}$ for the first probability above), Proposition \ref{SECOND=} (we deduce the estimate $1/(n-k)^{3/2}$ for the second probability) and Lemma \ref{LP1}.
\end{proof}

We now pass to the second part of Section \ref{sec:MR}, which is concerned with the expectation $\mathbb E[ \tau^{>r}>n; \phi(S_n)]$. 
\begin{theorem}
\label{MAIN>}
Assume that the hypotheses {\bf AA}, {\bf C} and {\bf M($2$)} hold. Then for any continuous function $\phi $ with compact support on  $\mathbb R$, one has
\begin{equation*}
     \lim_{n \to \infty} n^{3/2} \mathbb E[ \tau^{>r}>n; \phi(S_n)]= \iint_{D_r} \phi(x+y) U^{*+}({\rm d}x)a^{-}({\rm d}y)+ \iint_{D_r} \phi(x+y)a^{*+}({\rm d}x)U^{-}({\rm d}y),
\end{equation*}     
where 
$D_r:= \{(x, y) \in \mathbb R^2: 0\leq x \leq r, y\leq 0\}=[0,r]\times \mathbb R^{-}$.
\end{theorem}

We do not write the proof of Theorem \ref{MAIN>} in full details, for the three following reasons. First, it is similar to that of Theorem \ref{MAIN=}. We just emphasize the unique but crucial difference in the decomposition of the expectation $\mathbb E [\tau^{>r}>n; \phi(S_n)]$, namely:
\begin{equation}
\label{eq:decomposition_MAIN>}
     \mathbb E [\tau^{>r}>n; \phi(S_n)]=\sum_{0\leq k\leq n} \iint_{D_r} \phi(x+y)\mathbb P[\tau^{-}>k, S_k\in {\rm d}x] \mathbb P[\tau^{*+}>n-k, S_{n-k}\in {\rm d}y ].
\end{equation}
The second reason is that Theorem \ref{MAIN>} is equivalent to \cite[Theorem II.7]{LP1}. Indeed, the event $[\tau^{>r}>n]$ can be written as $[M_n\leq r]$, where $M_n = \max(0,S_1,\ldots ,S_n)$. Likewise, Proposition \ref{SECOND>} below on the asymptotics of $\mathbb P[\tau^{>r}>n]$ can be found in \cite{LP1}. Finally, Theorem \ref{MAIN>} is also proved in the recent article \cite{D}, see in particular Proposition 11.

\begin{proposition}
\label{SECOND>}
Assume that the hypotheses {\bf AA}, {\bf C} and {\bf M($2$)} hold. One has
\begin{equation}
     \mathbb P[\tau^{>r}>n] = \frac{\exp \kappa}{\sqrt{\pi n}} U^{*+}([0,r]) (1+o(1)),\quad n\to\infty.
\end{equation}
\end{proposition}

\begin{proof}
By \eqref{eq:decomposition_MAIN>}, the probability $\mathbb P[\tau^{>r}>n]$ may be decomposed as $\sum_{0\leq k\leq n} J_{n,k}( r    )$, with
\begin{align*}
     J_{n,k}( r )&= \iint_{D_r}\mathbb P[\tau^{-}>k, S_k\in {\rm d}x] \mathbb P[\tau^{*+}>n-k, S_{n-k}\in {\rm d}y ]\\
     & =\mathbb P[\tau^{-}>k, S_k\in [0,r]] \mathbb P[\tau^{*+}>n-k],
\end{align*}
where the domain $D_r$ is defined in Theorem \ref{MAIN>}. One concludes, using the following three facts. Firstly, by Remark \ref{rem-1}, one has $n^{3/2}\mathbb P[\tau^{-}>n, S_n\in[0,r]]\to a^{*+}([0,r])$ as ${n\to\infty}$. Secondly, by Equation \eqref{eq:ab_eps}, one has $\sqrt{n} \mathbb P[\tau^{*+}>n]\to e^\kappa/\sqrt{\pi}$ as ${n\to\infty}$. Thirdly, one has $\sum_{n\geq 0}\mathbb P[\tau^{-}>n,S_n\in[0,r]] = U^{*+}([0,r])$, also thanks to Remark \ref{rem-1}. 
\end{proof}

\begin{rem}
\label{rem-2}
Theorem \ref{MAIN=} (for which $r>0$) formally implies Theorem \ref{prop2} ($r=0$). To see this, it is enough to check that for $r=0$, the constant in the asymptotics of $\mathbb E[ \tau^{>r}=n; \phi(S_n)]$ coincides with the one in the asymptotics of $\mathbb E[ \tau^{*+}=n; \phi(S_n)]$. To that purpose, we first notice that for $r=0$, the domain $\Delta_r$ degenerates in $\{0\}\times \mathbb R^{*+}$. Furthermore, $U^{*+}(0)=1$ and $a^{*+}(0)=0$. Accordingly,
\begin{equation*}
     \iint_{\Delta_r} \phi(x+y) U^{*+}({\rm d}x)a^{-}({\rm d}y) +\iint_{\Delta_r} \phi(x+y)a^{*+}({\rm d}x)U^{-}({\rm d}y)= \int_{\mathbb R^{*+}} \phi(y)b^{*+}({\rm d}y).
\end{equation*}
In the right-hand side of the equation above, $\mathbb R^{*+}$ can be replaced by $\mathbb R^{+}$, as $b^{*+}(0)=0$. We then obtain the right constant in Theorem \ref{prop2}. Likewise, we could see that Theorem \ref{MAIN>} formally implies Theorem \ref{prop:classical_result}.
\end{rem}

%\begin{theorem}[\cite{LP1}]
%Assume that the hypotheses {\bf AA}, {\bf C} and {\bf M($2$)} hold. Then for any continuous function $\psi$ with compact support on $\mathbb R^+\times \mathbb R^+$, we have
%\begin{align*}
%     \lim_{n\to\infty}n^{3/2}\mathbb E[\psi(M_n,M_n-S_n)]&=\frac{1}{\sigma \sqrt{2\pi}} \iint_{\mathbb R^+\times \mathbb R^+}\psi(x,y)\lambda^{+}*U_{T^+}({\rm d}x)\overline{U}_{T^{'-}}({\rm d}y)\\
%     &+\frac{1}{\sigma \sqrt{2\pi}} \iint_{\mathbb R^+\times \mathbb R^+}\psi(x,y)U_{T^+}({\rm d}x)\lambda^{+}*\overline{U}_{T^{'-}}({\rm d}y).
%\end{align*}
%\end{theorem}

\section{Applications to random walks on $\mathbb R^{+}$ with non-elastic reflection at $0$}
\label{sec:RWabsorption}
\setcounter{equation}{0}

In this section we consider a sequence $(Y_i)_{i\geq 1}$ of i.i.d.\ random variables defined on a probability space $(\Omega, \mathcal T, \mathbb P)$, and we define the random walk $(X_n)_{n \geq 0}$ on $\mathbb R^+$ with non-elastic reflection at $0$ (or absorbed at $0$) recursively, as follows:
\begin{equation*}
     X_{n+1}:=\max(X_n+Y_{n+1}, 0), \quad \forall n\geq 0,
\end{equation*}
where $X_0$ is a given $\mathbb R^+$-valued random variable. The process $(X_n)_{n \geq 0}$ is a Markov chain on $\mathbb R^+$. We obviously have that for all $n\geq 0$, $X_{n+1} = f_{Y_{n+1}}(X_n)$, with
\begin{equation*}
     f_y(x) := \max (x+y,0),\quad \forall x,y\in\mathbb R.
\end{equation*}
The chain $(X_n)_{n \geq 0}$ is thus a random dynamical system; we refer the reader to \cite{PW1,PW2} for precise notions and for a complete description of recurrence properties of such Markov processes.

The profound difference between this chain and the classical random walk $(S_n)_{n\geq 0}$ on $\mathbb Z$ or $\mathbb R$ is due to the reflection at $0$. We therefore introduce the successive absorption times $({\bf a}_\ell)_{\ell\geq 0}$: 
\begin{align*}
{\bf a_0}&:= 0, \\
{\bf a}={\bf a}_1&:= \inf\{n>0: X_0+ Y_1+\cdots+Y_n<0\},\\
{\bf a}_{\ell}&:= \inf\{n>{\bf a}_{\ell-1}:  Y_{{\bf a}_{\ell-1}+1}+\cdots+Y_{{\bf a}_{\ell-1}+n}<0\},\quad \forall \ell\geq 2.
\end{align*}
Let us assume the first moment condition {\bf M($1$)} (i.e., that $\mathbb E[\vert Y_1\vert] <\infty$). If in addition $\mathbb E[ Y_1 ]>0$, the absorption times are not $\mathbb P$-a.s.\ finite, and in this case, the chain is transient. Indeed, one has $X_n\geq X_0+Y_1+\cdots +Y_n$, with $Y_1+\cdots +Y_n\to\infty$, $\mathbb P$-a.s. If $\mathbb E[ Y_1 ]\leq 0$, all the ${\bf a}_\ell$, $\ell\geq 1$, are $\mathbb P$-a.s.\ finite, and the equality $X_{{\bf a}_\ell}\mathbbm{1}_{\{{\bf a}_\ell<\infty\}}=0$, $\mathbb P$-a.s., readily implies that $(X_n)_{n\geq 0}$ visits $0$ infinitely often. On the event $[X_0 = 0]$, the first return time of $(X_n)_{n\geq 0}$ at the origin equals $\tau^{-}$. In the subcase $\mathbb E[ Y_1 ]=0$, it has infinite expectation, and $(X_n)_{n\geq 0}$ is null recurrent. If $\mathbb E[ Y_1 ]< 0$, one has $\mathbb E[\tau^{-}]<\infty$, and the chain $(X_n)_{n\geq 0}$ is positive recurrent.
In particular, when $\mathbb E[ Y_1 ]\geq 0$, for any $x \geq 0$ and any continuous function $\phi $ with compact support included in $\mathbb R^+$, one has
\begin{equation}
\label{eq:before_speed_convergence}
     \lim_{n \to \infty} \mathbb E [\phi(X_n)|X_0=x]=0.
\end{equation}     
We shall here focus our attention on the speed of convergence in \eqref{eq:before_speed_convergence}, by proving the following result:
\begin{theorem} 
\label{thm:speed_convergence}
Assume that the hypotheses {\bf AA}, {\bf C} and {\bf M($2$)} are satisfied. Then, for any $x \geq 0$ and any continuous function $\phi $ with compact support on  $\mathbb R^+$, one has
\begin{equation*}
     \lim_{n \to \infty} \sqrt{n}\mathbb E[  \phi(X_n)|X_0=x]= \frac{\widetilde\kappa}{ \sqrt{ \pi}} \int_{\mathbb R^+} \phi(t) U^+({\rm d}t),
\end{equation*}
where\footnote{We refer to Footnote \ref{footnote:Rim} for another expression of $\widetilde \kappa$.}
\begin{equation}
\label{eq:tildekappa}
     \widetilde \kappa:=\exp\left( \sum_{n\geq1} {\mathbb P[S_n<0]-1/2\over n}\right).
\end{equation}

If $\mathbb E[Y_1]>0$ and if furthermore ${\bf AA}$ and {\bf M($\exp^-$)} hold,\footnote{In fact, it would be sufficient to assume that $\mathbb E[e^{\gamma Y_1}]<\infty$ for $\gamma$ belonging to some interval $[a, 0]$, if $[a, 0]$ is such that the convex function $\gamma\mapsto \mathbb E[e^{\gamma Y_1}]$ reaches its minimum at a point $\gamma_0\in ]a, 0[$.} there exists $\rho= \rho(\mu)\in ]0, 1[$ and a positive constant $C(\phi)$ (which can be computed explicitly) such that
\begin{equation*}
     \lim_{n \to \infty} {n^{3/2}\over \rho^n} \mathbb E[  \phi(X_n)|X_0=x]= C(\phi).
     \end{equation*}
\end{theorem}

\begin{proof}
%Notice that $X_{{\bf a}_\ell}\mathbbm{1}_{\{{\bf a}_\ell<\infty\}}=0$, $\mathbb P$-a.s., which explains why the definition  of ${\bf a_1}$ differs from the one of ${\bf a}_\ell$, $\ell\geq 2$.\footnote{\bf Ca vaudrait peut-etre le coup de mettre ca en footnote.}
We first assume that $X_0=0$. On the event $[T_\ell^{*-}\leq n<T_{\ell+1}^{*-}]$, one has that $X_n= S_n-S_{T_\ell^{*-}}$. It readily follows that
 \begin{eqnarray*}
 \mathbb E[ \phi(X_n)\hspace{-4mm}&|&\hspace{-4mm}X_0=0]\\
 &=&\hspace{-2mm} \sum_{\ell\geq 0} \mathbb E[{\bf a}_\ell\leq n<{\bf a}_{\ell+1}; \phi(X_n)|X_0=0]\\
 &=&\hspace{-2mm} \sum_{\ell\geq 0}  \mathbb E[T_\ell^{*-}\leq n<T_{\ell+1}^{*-}; \phi(X_n)|X_0=0]\\
 &=&\hspace{-2mm} \sum_{\ell\geq 0}  \mathbb E[T_\ell^{*-}\leq n<T_{\ell+1}^{*-}; \phi(S_n-S_{T_\ell^{*-}})]\\
 &=&\hspace{-2mm} \sum_{\ell\geq 0} \sum_{0\leq k\leq n} \mathbb E[T_\ell^{*-}=k, Y_{k+1}\geq 0, \ldots,  Y_{k+1}+\cdots +Y_{n}\geq 0; \phi(Y_{k+1}+\cdots +Y_{n})]\\
 &=&\hspace{-2mm} \sum_{0\leq k\leq n} \left( \sum_{\ell \geq0}\mathbb P[T_\ell^{*-}=k]\right)\mathbb E[ Y_{k+1}\geq 0, \ldots,  Y_{k+1}+\cdots +Y_{n}\geq 0; \phi(Y_{k+1}+\cdots +Y_{n})].
\end{eqnarray*}
Using the fact that for any $k\geq 0$, the events $[T_\ell^{*-}=k]$, $\ell\geq 0$, are pairwise disjoint together with the fact that $\mathcal L (Y_1, \ldots, Y_n)=\mathcal L(Y_n, \ldots, Y_1)$, one gets
\begin{equation*}
     \sum_{\ell\geq 0}\mathbb P[T_\ell^{*-}=k]=\mathbb P[\exists \ell\geq 0,T_\ell^{*-}=k]= \mathbb P[S_k<0, S_k<S_1, \ldots, S_k<S_{k-1}]=\mathbb P[\tau^{+}>k],
\end{equation*}
which in turn implies that
\begin{equation}
\label{partant de 0}
     \mathbb E[ \phi(X_n)|X_0=0]=\sum_{0\leq k\leq n}  \mathbb P[\tau^{+}>k] \mathbb E[ \tau^{*-}>n-k; \phi(S_{n-k})].
\end{equation}
  
The situation is more complicated when the starting point is $x \geq 0$. In that case, one has the decomposition
\begin{equation}
\label{partant de x}
     \mathbb E [ \phi(X_n)|X_0=x]=E_1(x, n)+E_2(x, n),
\end{equation} 
with $E_1(x, n):= \mathbb E[ {\bf a} >n; \phi(X_n)|X_0=x]$ and $E_2(x, n):= \mathbb E[ {\bf a} \leq n; \phi(X_n)|X_0=x]$. From the definition of ${\bf a}$, one gets $E_1(x, n)=\mathbb E[ \tau^{<-x}>n; \phi(x+S_n)]$. Similarly, by the Markov property and the fact that  $X_{\bf a}=0$, $\mathbb P$-a.s., one may write
\begin{equation*}
      E_2(x, n)= \sum_{0\leq \ell\leq n}   \mathbb P[\tau^{<-x}=\ell] \mathbb E[\phi(X_{n-\ell})|X_0=0].
 \end{equation*}
 
\subsection*{The centered case}
 
We first assume that  hypotheses {\bf AA} and {\bf M($2$)} are satisfied and that the $(Y_i)_{i\geq 1}$ are centered (hypothesis $\bf C$).  In this case, by fluctuation theory of centered random walks, one gets $\mathbb P[{\bf a}_\ell<\infty]=1$ for any $\ell \geq 0$ and any initial distribution $\mathcal L(X_0)$.
 
We first consider the case when $X_0=0$ and we use the identity \eqref{partant de 0}. By \cite[Theorem II.2]{LP1} (see also how \eqref{eq:ab_eps} is obtained), one gets
\begin{equation*}
    \lim_{n\to \infty} \sqrt{n}\mathbb P[\tau^{+}>n]={\widetilde\kappa\over \sqrt{\pi}},
\end{equation*}
with $\widetilde \kappa$ defined in \eqref{eq:tildekappa}. On the other hand, by Remark \ref{rem-1} in Section \ref{sec:FR} we know that%\footnote{\bf Explain notation for $a^+(\phi)$?}
\begin{equation*} 
     \lim_{n\to \infty}n^{3/2}\mathbb E[ \tau^{*-}>n; \phi(S_{n})] =a^+(\phi).
\end{equation*}
We conclude, setting $c_n:=\mathbb P[\tau^+>n$, $d_n:=\mathbb E[\tau^{*-}>n; \phi(S_n)]$, thus $c :=\widetilde \kappa/{\sqrt{\pi}}$ and $D:= \sum_{n\geq 0}\mathbb E[ \tau^{*-}>n; \phi(S_{n})]=U^+(\phi)$, in Lemma \ref{Bruijn2}.
%\begin{lemma}[\cite{I}]
%Let $(a_n)_{n \geq 0}$ and $(b_n)_{n \geq 0}$ be two sequences of non-negative real numbers such that 
%\begin{itemize}
%\item $ \lim_{n \to \infty} \sqrt{n} a_n= a \in \mathbb R^{*+}$;
%\item$ B:=  \sum_{n\geq 0} b_n<\infty$;
%\item  the sequence $(nb_n)_{n \geq 0}$ is bounded.
%\end{itemize}
%Then $\lim_{n \to \infty} \sqrt{n}\sum_{0\leq k\leq n}a_k b_{n-k}= a B$.
% \end{lemma}
  
In the general case (when $X_0=x$), we use identity \eqref{partant de x}. By the results of Section \ref{sec:MR} (Theorem \ref{MAIN>} with $\tau^{<-x}$ instead of $\tau^{>r}$), one gets $E_1(x, n)=  O(n^{-3/2})$.\footnote{Notice that in the preceding formula, $O(n^{-3/2})$ depends on $x$.} On the other hand, by the Markov property, since $X_{\bf a}=0$, $\mathbb P$-a.s., one has
\begin{align*}
E_2 (x, n)&= \sum_{0\leq k\leq n}   \mathbb E[ {\bf a}=k; \phi(X_n)|X_0=x]
\\
&=\sum_{0\leq k\leq n}  \mathbb P[{\bf a}=k|X_0=x]
 \mathbb E[\phi(X_{n-k})|X_0=0]
\\
&=\sum_{0\leq k\leq n} \mathbb P[\tau^{<-x}=k]
 \mathbb E[\phi(X_{n-k})|X_0=0].
\end{align*}
Recall that $\lim_{n \to \infty} \sqrt{n} \mathbb E[\phi(X_{n})|X_0=0]= (\widetilde \kappa/ \sqrt{\pi})U^+(\phi)$; on the other hand, it follows from Proposition \ref{SECOND=} (with $\tau^{<-x}$ instead of $\tau^{>r}$) that $(n\mathbb P[\tau^{<-x}=n])_{n \geq 0}$ is bounded. Furthermore, $\sum_{n \geq 1} \mathbb P[\tau^{<-x}=n]=\mathbb P[\tau^{<-x}<\infty]=1$. One may thus apply Lemma \ref{Bruijn2}, which yields
\begin{equation*}
     \lim_{n \to \infty} \sqrt{n} \mathbb E[\phi(X_{n-k})|X_0=x]=\lim_{n \to \infty} \sqrt{n} E_2(x, n)=\frac{\widetilde \kappa}{\sqrt{\pi}} U^+(\phi).
\end{equation*}

\subsection*{The non-centered case}
Hereafter, we assume that hypotheses $\bf M($1$)$, {\bf M($\exp^-$)} and $\bf AA$ hold, and that in addition $\mathbb E[Y_1]>0$. We use the standard relativisation procedure that we now recall: the function 
\begin{equation*}
     \widehat{\mu}(\gamma):= \mathbb E[e^{\gamma Y_1}]
\end{equation*}      
is well defined on $\mathbb R^-$, tends to $\infty$ as $\gamma \to -\infty$, and has derivative $\mathbb E[Y_1]>0$ at $0$. It thus achieves its minimum at a point $\gamma_0<0$, and we have $\rho:= \widehat{\mu}(\gamma_0)\in ]0, 1[$. The measure 
\begin{equation*}
     \widetilde{\mu}({\rm d}x):= (1/ \rho)e^{\gamma_0 x}\mu({\rm d}x)
\end{equation*}
is a probability on $\mathbb R$. Furthermore, if $(\widetilde{Y}_i)_{i\geq 1}$ is a sequence of  i.i.d.\ random variables with law $\widetilde{\mu}$ and $(\widetilde{S}_n)_{n \geq 1}$ is the corresponding random walk on $\mathbb R$ starting from $0$, one gets
 \begin{equation*}
     \mathbb E[\varphi(Y_1, \ldots, Y_n)]= \rho^n\mathbb E[\varphi(\widetilde{Y}_1, \ldots, \widetilde{Y}_n)e^{-\gamma_0 \widetilde{S}_n}]
 \end{equation*}
 for any $n \geq 1$ and any bounded test Borel function $\varphi: \mathbb R^n\to \mathbb R$. Denoting by $\widetilde{\tau}^+$ and $\widetilde{\tau}^{*-}$ the first entrance times of $(\widetilde{S}_n)_{n \geq 1}$ in $\mathbb R^+$ and $\mathbb R^{*-}$, respectively, we may thus write \eqref{partant de 0} as 
\begin{eqnarray*}
\mathbb E[\phi(X_n)|X_0=0] = \rho^n
\sum_{0\leq k\leq n} \mathbb E[\widetilde{\tau}^+>k; e^{-\gamma_0 \widetilde{S}_{k}}] \mathbb E[\widetilde{\tau}^{*-}>n-k; \phi(\widetilde{S}_{n-k})e^{-\gamma_0 \widetilde{S}_{n-k}}],
\end{eqnarray*}
and by Lemma \ref{LP1} the sequence $(({n^{3/2}/ \rho^n})\mathbb E[\phi(X_n)|X_0=x])_{n\geq 0}$ converges to some constant $C(\phi)>0$.

Following the same way, for any $x \geq 0$ one can decompose as above $\mathbb E[\phi(X_n)|X_0=x]$ as $E_1(x, n)+E_2(x, n)$, with 
\begin{align*}
E_1(x, n)&=\rho^n
\mathbb E[\widetilde{\tau}^{<-x}>n; \phi(\widetilde{S}_{n})e^{-\gamma_0 \widetilde{S}_{n-k}}],\\
 E_2(x, n)&=  
 \sum_{0\leq k\leq n} \rho^k 
\mathbb E[\widetilde{\tau}^{<-x}=k; e^{-\gamma_0 \widetilde{S}_{k}}] \mathbb E[\phi(X_n)|X_0=0].
\end{align*}
One concludes using Section \ref{sec:MR} (Theorem \ref{MAIN=} with $\tau^{<-x}$ instead of $\tau^{>r}$) for the behavior of the sequence $(\mathbb E[\widetilde{\tau}^{<-x}=n; e^{-\gamma_0 \widetilde{S}_{n}}])_{n\geq 0}$ and the previous estimation for the behavior of $(\mathbb E[\phi(X_n)|X_0=0] )_{n \geq 0}$.
\end{proof}

\section{Local limit theorems and links with results by Denisov and Wachtel}
\label{sec:DW}
\setcounter{equation}{0}
Hereafter, we shall assume that {\bf AA}$(\mathbb Z)$ holds; in particular, the random walk $(S_n)_{n\geq 0}$ is $\mathbb Z$-valued. Taking $\phi(S_n)= \mathbbm{1}_{\{S_n=i\}}$, Theorem \ref{MAIN>} immediately leads to:
\begin{cor}
\label{cor:llt}
Assume that the hypotheses {\bf AA}$(\mathbb Z)$, {\bf C} and {\bf M($2$)} hold. Then for $i\leq r$,
\begin{equation*}
     \mathbb P[\tau^{>r}>n,S_n = i]= \frac{Z(r,i)}{n^{3/2}}(1+o(1)),\quad n\to\infty,
\end{equation*}
where we have set
\begin{equation}
\label{eq:our_harmonic_function}
     Z(r,i)=\sum_{\max\{i,0\}\leq k\leq r} [a^{-}(i-k)U^{*+}(k)+U^{-}(i-k)a^{*+}(k)].
\end{equation}
\end{cor}

It is worth noting that the definition of $a^{-}$ implies that for $y\in\mathbb Z^{*+}$, $a^-(y)=0$, and for $y\in\mathbb Z^{-}$, 
\begin{equation*}
     a^-(y) = \frac{1}{\sigma \sqrt{2\pi}}\sum_{n\geq0}\mathbb E[\tau^{*+}>n; S_n\in [y,0]].
\end{equation*}
Likewise, for $y\in\mathbb Z^{-}$, $a^{*+}(y)=0$, and for $y\in\mathbb Z^{*+}$, 
\begin{equation*}
     a^{*+}(y) = \frac{1}{\sigma \sqrt{2\pi}} \sum_{n\geq0}\mathbb E[\tau^{-}>n; S_n\in ]0,y]].
\end{equation*}%\footnote{\bf Pour $a^{*+}(y)$ c'est bien ça ?}

\begin{rem}
\label{rem-3}
Using these facts and similar remarks for the potentials $U^{*+}$ and $U^{-}$, we obtain that the quantity \eqref{eq:our_harmonic_function} can also be written as a sum of two convolution terms:
\begin{align}
\label{eq:our_harmonic_function_bis}
     Z(r,i)&=\sum_{-\infty< k< r} [a^{-}(i-k)U^{*+}(k)+U^{-}(i-k)a^{*+}(k)]\\
             &=\sum_{-\infty< k< \infty} [a^{-}(i-k)U^{*+}(k)\mathbbm{1}_{\{k\leq r\}}+U^{-}(i-k)a^{*+}(k)\mathbbm{1}_{\{k\leq r\}}].
\label{eq:our_harmonic_function_ter}             
\end{align}     
\end{rem}

In the remaining of this section we compare the local limit theorem of Corollary \ref{cor:llt} with the one in \cite{DW}. All results taken from \cite{DW} make the assumptions that the $(Y_i)_{i\geq 1}$ have moments of order $2+\epsilon$, with $\epsilon>0$. To state the local limit theorem \cite[Theorem 7]{DW}, %we shall need some extra notations.%\footnote{\bf Dans \cite{DW}, il y a l'ambiguite sur la nature du cone. Ce nest pas clair si c'est un ferme ou un ouvert. Par exemple page 18 section 2.4 ils prennent $(0,\infty)$, done ca suggere que le cone est toujours ouvert, tandis que dans leur section quart de plan (example 3 page 5) ils prennent $\mathbb {R^{+}}^2$. Mais ce n'est peut-etre pas crucial. \textcolor{red}{Je choisis de dire que leur cone est ouvert}.} 
we need to introduce the function (see \cite[Section 2.4]{DW} for more details)
\begin{equation}
\label{eq:their_harmonic_function}
     V(x):=-\mathbb E[S_{\tau^{\leq-x}}]=-\mathbb E[S_{\tau^{<-x+1}}].
\end{equation}
This function is positive %\footnote{{\bf Faire un peu attention ici aux $>,\geq ,<,\leq$.}} \footnote{\bf On pourrait detailler un peu ici: par exemple, dire clairement pourquoi la fonction est positive sur tel ou tel domaine.}
on $\mathbb R^{+}$ and is harmonic for the random walk $(S_n)_{n\geq 0}$ killed when reaching $\mathbb R^{-}$; it means that for $x>0$,
\begin{equation*}
     \mathbb E[V (x + Y_1); \tau^{<-x}> 1] = V (x).
\end{equation*} 
%Further, it is proved in \cite{DW} that 
%\begin{equation*}
%     V(x)=x+\lim_{n\to\infty}\mathbb E[S_n; \tau^{<-x}>n].
%\end{equation*}
Define $V'$ as the harmonic function for the random walk with increments $(-Y_i)_{i\geq 1}$ with the same construction as \eqref{eq:their_harmonic_function}. We have the following result:

\begin{theorem}[\cite{DW}]
\label{thm:reformulation-DW}
Assume that the hypotheses {\bf AA}$(\mathbb Z)$, {\bf C} and {\bf M($2+\epsilon$)} hold. Then for $i\leq r$,
\begin{equation*}
     \mathbb P[\tau^{>r}>n,S_n = i]=\frac{1}{\sigma} \sqrt{\frac{2}{\pi}} \frac{V'((r+1)/\sigma)V((r+1-i)/\sigma)}{n^{3/2}}(1+o(1)),\quad n\to\infty.
\end{equation*}
\end{theorem}

\begin{proof}
Theorem 7 in \cite{DW} states that if $(\widetilde S_n)_{n\geq 0}$ is a random walk on a lattice $h\mathbb Z$ starting from $0$ and with increments $(\widetilde Y_i)_{i\geq 0}$ having a variance equal to $1$, the following local limit theorem holds:
\begin{equation*}
     \mathbb P[x+\widetilde S_n = y,\tau^{\leq-x}>n]=h  \sqrt{\frac{2}{\pi}} \frac{V(x)V'(y)}{n^{3/2}}(1+o(1)),\quad n\to\infty.
\end{equation*}
%with $\varrho=(\sqrt{2/\pi})/4$ and $H_0=1$.\footnote{{\bf Il y a une contradiction entre le memoire d'Oanh et l'article \cite{DW}. Oanh a trouve $\sqrt{2/\pi}$ pour $\varrho$. A qui faire confiance ?}} 
Applying this result to the random walk $(\widetilde S_n)_{n\geq 0}:=(-S_n/\sigma)_{n\geq 0}$, and letting $x:=(r+1)/\sigma$ and $y:=(r+1-i)/\sigma$, we obtain Theorem \ref{thm:reformulation-DW}.
\end{proof}

By Corollary \ref{cor:llt} and Theorem \ref{thm:reformulation-DW}, we must have
\begin{equation}
\label{eq:two_constants_coincide}
     Z(r,i) = \frac{1}{\sigma} \sqrt{\frac{2}{\pi}} V'((r+1)/\sigma)V((r+1-i)/\sigma).
\end{equation}
However:
\begin{question}
It is an open problem to show by a direct computation that \eqref{eq:two_constants_coincide} holds.
\end{question}
To conclude Section \ref{sec:DW}, we prove \eqref{eq:two_constants_coincide} for the simple random walk, with probabilities of transition $\mathbb P[Y_i=-1]=\mathbb P[Y_i=1]=p$ and $\mathbb P[Y_i=0]=1-2p$. In this case the harmonic functions have the simple form $V(x) = V'(x) = x$, and obviously $\sigma=\sqrt{2p}$. We deduce that the constant in Theorem \ref{thm:reformulation-DW} is
\begin{equation}
\label{eq:first_constant}
     \frac{(r+1)(r+1-i)}{2p^{3/2}\sqrt{\pi}}.
\end{equation} 
To compute $Z(r,i)$, we start from the formulation \eqref{eq:our_harmonic_function}, where we assume that $i\geq 0$ (the computation for $i<0$ would be similar). We recall that for the simple random walk one has $U^{*+}(k) =\mathbbm{1}_{\{k\geq 0\}}$ and $U^{-}(k)=\mathbbm{1}_{\{k\leq 0\}}/p$. Then for $k\leq 0$, $a^{-}(k)=(|k|+1)/(p{\sigma}\sqrt{2\pi})$ and for $k\geq 0$, $a^{*+}(k) = k/(\sigma\sqrt{2\pi})$. We deduce that 
\begin{equation*}
     Z(r,i) = \frac{1}{p\sigma\sqrt{2\pi}}\sum_{i\leq k\leq r} [(k-i+1)+k].
\end{equation*}
It is then an easy exercise to show that $Z(r,i)$ equals \eqref{eq:first_constant}.

%Notre idee pour comparer \cite{DW} et notre resultat est de partir de la fonction harmonique $V$ en \eqref{eq:their_harmonic_function} et de l'exprimer avec des potentiels. Pour ca, il semble que le lemme suivant soit particulierement utile :

%\begin{lemma}
%\label{lemma:law_S_tau}
%Assume that the hypotheses {\bf AA}$(\mathbb Z)$, {\bf C} and {\bf M($2$)} hold. For $-y<-x\leq 0$, one has
%\begin{equation}
%     \mathbb P[S_{\tau^{<-x}}=-y] = \sum_{0\leq w\leq x} U^{*-}(-w)\mu^{*-}(w-y).
%\end{equation}
%\end{lemma}
%
%\begin{proof}
%One has
%\begin{align*}
%     \mathbb P[S_{\tau^{<-x}}=-y] &= \sum_{n\geq 1} \mathbb P[\tau^{<-x}=n,S_{n}=-y] \\
%     &  =\sum_{n\geq 1}\sum_{0\leq k\leq n}  \mathbb P[\exists \ell\geq 0,T^{*+}_\ell=k, S_k\geq -x, n-k=\tau^{*+}_{\ell+1}, S_n=-y] 
%\end{align*}
%A FINIR!
%\end{proof}

\section*{Acknowledgments}
\setcounter{equation}{0}
We wish to thank Nguyen Thi Hoang Oanh for useful discussions concerning Section \ref{sec:DW}. We are grateful to Vitali Wachtel for pointing out the reference \cite{D}. Finally, we thank an anonymous referee for useful comments and suggestions.

\end{document}